\documentclass{svjour3}                     
\usepackage{graphicx}

\usepackage{amsmath}
\usepackage{amssymb}
\usepackage{color}
\usepackage{bbold}
\newcommand{\D}{\mathrm{d}}
\newcommand{\R}{\mathbb{R}}

\newcommand{\dst}{\displaystyle}

\newtheorem{assumption}{Assumptions}
\begin{document}

\title{Mass concentration in a nonlocal model of clonal selection
}
\subtitle{}


\author{J.-E.~Busse \and
		P.~Gwiazda \and
		 A.~Marciniak-Czochra }


\institute{P.Gwiazda \at Institute of Applied Mathematics and Mechanics,  University of Warsaw,  ul. Banacha 2, Warsaw 02-097, Poland \and 
A. Marciniak-Czochra \at Institute of Applied Mathematics, Interdisciplinary Center for Scientific Computing (IWR) and  BIOQUANT, University of Heidelberg, Im Neuenheimer Feld 294, 69120 Heidelberg, Germany \and 
J.E. Busse \at Institute of Applied Mathematics, BIOQUANT, University of Heidelberg, Im Neuenheimer Feld 294, 69120
}

\date{Received:/ Accepted: date}

\maketitle

\begin{abstract}
Self-renewal is a constitutive property of stem cells. Testing the
cancer stem  cell hypothesis requires investigation of the impact of self-renewal on
cancer  expansion.  To better understand this impact, we 
propose a mathematical model describing the dynamics of a continuum of cell clones structured by the self-renewal 
potential. The model is an extension of the finite multi-compartment models of interactions between normal and cancer cells in acute leukemias. It takes a form of a system of integro-differential equations with a nonlinear and nonlocal coupling which describes regulatory feedback loops of cell proliferation and differentiation.  We show that this coupling leads to mass concentration in points corresponding to the maxima of the self-renewal potential and the solutions of the model tend asymptotically to Dirac measures multiplied by positive constants.
Furthermore,  using a Lyapunov function constructed for the finite dimensional counterpart of the model,  we prove that the total mass of the solution converges to a globally stable equilibrium. Additionally, we show stability of the model  in the space of positive Radon measures equipped with the flat metric (bounded Lipschitz distance). Analytical results are illustrated by numerical simulations.
\end{abstract}

\section{Introduction}
\label{intro}

\noindent  This paper is devoted to the analysis of a structured population model describing clonal evolution of acute leukemias. Leukemia is a disease of the blood production system leading to an extensive expansion of malignant cells that are non-functional and cause an impairment of blood regeneration. Recent experimental evidence indicates that cancer cell populations are composed of multiple clones consisting of genetically identical cells \cite{Ding} and maintained by cells with stem-like properties \cite{Bonnet,Hope}. Many authors have provided evidence for heterogeneity of leukemic stem cells (LSC) attempting to identify their characteristics; for review see Ref. \cite{Lutz}. Heterogeneity is further supported by the results of gene sequencing studies \cite{Ding,Ley}.  However, it was shown in these studies that a limited number of clones contribute to the total leukemic cell mass. At most 4 contributing clones were detected in the case of acute myeloid leukemia (AML) and at most 10 in the case of acute lymphoblastic leukemia (ALL)  \cite{Ding,Lutz}.  Moreover, in most cases of ALL,  the clones dominating the relapse have already been present at the diagnosis but undetectable by the routine methods  \cite{VanDelft,Choi,LutzLeuk}.  Due to a quiescence, a very slow cycling or other intrinsic mechanisms \cite{LutzLeuk,Choi}, these clones may survive chemotherapy and eventually expand  \cite{LutzLeuk,Choi}. This implies that the main mechanism of relapse in ALL might be selection of existing clones and not acquisition of therapy-specific mutations \cite{Choi}. Similar mechanisms have been described in AML \cite{Ding,Jan}.  Based on these findings the evolution of malignant cells can be interpreted as a selection process  for properties that enable cells to survive the treatment {and to expand efficiently.} The mechanisms of the underlying process and its impacts on the disease dynamics and on the response of cancer cells to chemotherapy are not understood.  Gene sequencing studies allow deciphering the genetic relations among different clones; nevertheless the impact of many detected mutations on cell behaviour remains unclear \cite{Ding}. The multifactorial nature of the underlying processes severely limits the intuitive interpretation of the experimental data.  \\

\noindent To investigate the impact of cell properties on the multi-clonal composition of leukemias and to elucidate the possible mechanisms of the clonal selection suggested by the experimental data, a multi-compartmental model was proposed and studied numerically in Ref. \cite{StiehlBaranHoMarciniak}.   It assumes the form of the following system of ordinary differential equations,
\begin{eqnarray}\label{sys:multi}
\frac{d}{dt}c_1(t)&=&\big(2a^cs(t)-1\big)p^cc_1(t),\nonumber\\
\frac{d}{dt}c_2(t)&=&2\big(1-a^c s(t)\big)p^cc_{1}(t)-d^c_2c_2(t),\nonumber\\
\frac{d}{dt}l^1_1(t)&=&\big(2a^{l^1} s(t)-1\big)p^{l^1}l^1_1(t),\nonumber\\
\frac{d}{dt}l^1_2(t)&=&2\big(1-a^{l^1} s(t)\big)p^{l^1}l^1_{1}(t)-d^{l^1}_2l^1_2(t),\nonumber\\
\vdots&\vdots&\vdots
\end{eqnarray}
\begin{eqnarray}
\frac{d}{dt}l^n_1(t)&=&\big(2a^{l^n} s(t)-1\big)p^{l^n}l^n_1(t),\nonumber\\
\frac{d}{dt}l^n_2(t)&=&2\big(1-a^{l^n} s(t)\big)p^{l^n}l^n_{1}(t)-d^{l^n}_2l^n_2(t),\nonumber\\\nonumber\\
s(t)&=&\frac{1}{1+K^cc_2(t)+K^l\sum_{i=1}^n l^i_2(t)},\nonumber
\end{eqnarray}
with nonnegative initial data.\\

\noindent The model  describes time dynamics of a healthy cell line, denoted by $c_j$, $j=1,2$ and of $n$ clones of leukemic cells $l^i_j$, for $j=1,2,$ and $i=1,...,n$, at time $t$. Each population consists 
of two different cell types, proliferating and non-proliferating, denoted by $j=1$ and $j=2$, respectively.  This two-compartment model is a simplification of the more realistic model with multiple differentiation stages; see Ref. \cite{MCSHJW,StiehlHoMarciniak} for an introduction to the model and its application to the healthy hematopoiesis;  Ref. \cite{Getto,Nakata,Stiehl} for its analysis;
and Ref. \cite{Doumic} for a continuous-structure extension. This model can be viewed as a structured population model with a discrete structure describing two differentiation stages and $n+1$ cell types.\\

\noindent Parameters $p^c>0$ and $p^{l^i}>0$ denote the proliferation rate of the healthy cells and the cells in the leukemic clone $i$, respectively, and $a^c$ and $a^{l^i}$ are the corresponding maximal  fractions of self-renewal, which depend on the proportion of
symmetric and asymmetric cell divisions in the respective population. More precisely, the self-renewal fractions $0<a^c<1$ and $0<a^{l^i}<1$ are the fractions of the progeny cells that remain in the compartment of proliferating cells. Consequently, $(1- a^{c})$ and $(1- a^{l^i})$ are fractions of 
the dividing cells that differentiate and become non-proliferating. By  $d_2^c>0$ and $d_2^{l^i}>0$ we denote the clearance rate of the non-proliferating healthy cells and the cells in the $i$-th leukemic clone, respectively. \\

\noindent The model is based on the assumption that leukemic clones and their normal counterparts respond to a hematopoietic feedback signalling and compete for  signalling factors (cytokines). We assume that the feedback signal, $s(t)$, decreases if the number of non-proliferating cells increases. Derivation of such nonlinear feedback loop was proposed in Ref. \cite{MCSHJW}. It  is based on a Tikhonov-type quasi-stationary approximation of dynamics of the extracellular signalling molecules, such as the G-CSF cytokine, which are secreted by specialised cells at a constant rate and degraded by a receptor-mediated endocytosis. {  Following the evidence from clinical trials that the mature granulocytes mediate clearance of G-CSF  \cite{Layton}, we assume that dynamics of the signalling molecules depends on the number of non-proliferating cells. This assumption has been also supported by studies of receptor expression showing that the mature cells express significantly more receptors than the cells in bone marrow \cite{Shinjo}. Taking into account these observations, we obtain a model with a nonlinear coupling depending on the level of non-proliferating cells.}\\

\noindent Numerical simulations of model \eqref{sys:multi} suggest that cells with a superior self-renewal potential, i.e. a maximum value of the parameter $a$, reflecting the probability that a daughter cell has the same properties and fate as its parent cell,  have an advantage in comparison to their competitors, which leads to the expansion of this cell 
subpopulation  \cite{StiehlBaranHoMarciniak}. The phenomenon was shown analytically solely in the case of two competing populations, a healthy and a cancerous cell line \cite{StiehlMarciniakMMNP}.\\

\noindent To elucidate further mechanisms of clonal selection, we propose an infinitely dimensional extension of the multi-compartment model  \eqref{sys:multi}. We introduce a continuous variable $x\in \Omega$ that represents the expression level of  genes (yielding a phenotype) influencing self-renewal properties of the cells. 
It leads to a system of integro-differential equations describing dynamics of a structured population with the continuum of cell clones and the two-compartment differentiation structure. Cells in Population 1 (dividing cells) proliferate and may self-renew or differentiate into Population 2 cells (differentiated cells). Population 2 cells do not proliferate and die after an exponentially distributed lifetime, as depicted in  Fig. \ref{scheme}. Cells in both populations are stratified by a structure variable $x$. We assume that the self-renewal parameter depends on $x$, i.e. the parameter $a$ becomes a function $a(x)$. These assumptions lead to the model
 \begin{eqnarray}\label{sys:reg0}
					\frac{\partial}{\partial t} u_1(t,x) &=& \left(2a(x)s(t)-1\right)pu_1(t,x), \nonumber\\
					\frac{\partial}{\partial t} u_2(t,x) &=& 2\left(1-a (x)s(t) \right)pu_1(t,x) - d u_2(t,x), \\
					u_1(0,x) &=& u_1^0(x), \nonumber\\
					u_2(0,x)& =& u_2^0(x). \nonumber
\end{eqnarray}
Assuming $s(t) = 1/\left(1+K\int\limits_{\Omega} u_2 (t,x)\D x\right)$, we obtain a nonlocal and nonlinear coupling of the two equations.\\
\begin{figure}[htpb] 
\begin{center}
\includegraphics[width=0.96 \textwidth]{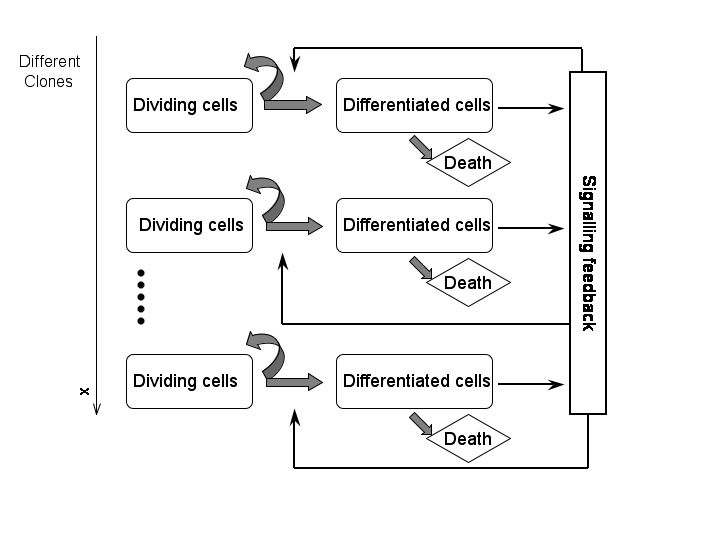}
\end{center}
\caption{Schematic representation of model \eqref{sys:reg0}, consisting of  two compartments corresponding to undifferentiated cells (dividing cells) and
mature cells (differentiated cells). Undifferentiated cells (stem cells and early progenitors) divide symmetrically or asymmetrically. Accordingly, they produce cells of the same type (self-renewal) and mature cells (differentiation). Mature cells do not divide and they die after an exponentially distributed lifetime. The cells in each compartment are heterogenous. They are stratified by a structure variable $x$ that represents the expression level of genes (yielding a phenotype and eg. influencing the self-renewal properties of the cells). Self-renewal and differentiation of cells are regulated by a cytokine feedback which, in
turn, depends on the total count of differentiated cells.} \label{scheme}
\end{figure}

\noindent Our approach is motivated by the theory of selection of the most fit variants in adaptive evolution.  Cells with different mutational variants might have different growth properties allowing them to expand more efficiently.  The phenomenon can be understood as an example of closely related to the process of Darwinian evolution. In our particular case, certain rare mutants may have positive growth rates and be selected in environments that otherwise result in extinction.  In other words, cells with a  fitness advantage expand and dominate dynamics of the population leading to extinction of the other cell clones. The model proposed belongs to the class of selection models exhibiting a mass concentration effect, similar to those presented in the books \cite{perthame} and {\cite{B}}.\\

\noindent In the current work, we do not model mutation events. Instead, motivated by the experimental findings described earlier in Ref. \cite{LutzLeuk,Choi}, we aim to understand which aspects of the dynamics of leukemias can be  explained by the selection alone. It is interesting, since the relapse caused by an expansion of a clone that could not be detected at diagnosis due to the limited sensitivity of detection methods, can be misinterpreted as a mutational event \cite{Choi}. {  A computational model of the AML with mutations was proposed in Ref. \cite{StiehlBaranHoMarciniak}.  Following the biological evidence \cite{Jan12}, it was assumed that new LSC clones were formed due to mutations occurring in LSCs or  due to the influx from the so-called preleukemic cells at a rate modelled by a time inhomogeneous Poisson process. At each point of the Poisson process a new clone with random cell properties was added to the system.
 Simulations of that model demonstrate that leukemic cell properties at diagnosis and at relapse are comparable to the scenario without mutations.
 Introducing mutations to the continuous models is known to make asymptotic analysis more complicated, and therefore we do not consider this aspect in the current paper.}\\

\noindent The mathematical angle of our study is analysis of the nonlocal effects and development of singularities in the solutions of the integro-differential equations.
We show that the solutions of system \eqref{sys:reg0} may tend to Dirac measures concentrated in points with the largest value of the self-renewal potential. Such dynamics can be interpreted in the terms of selection, which causes convergence of the heterogeneous initial data to a stationary solution with the mass localised on a set of measure zero. Convergence then holds in the weak$*$ topology of Radon measures. { Considering the space of positive Radon measures with  a suitable metric allows formulating the result on convergence of solutions to a stationary measure in the terms of the metric instead of the weak$*$ convergence of Radon measures}. We apply the flat metric (bounded Lipschitz distance), which has proven to be useful in the analysis of a variety of transport equations models, for example to study Lipschitz dependence of solutions of nonlinear structured population models on the model parameters and initial data \cite{GLMC,GMC,CCGU};  see Appendix for the definition and properties of the flat metric. \\

\noindent Similar results have been recently shown for scalar equations including diffusion; see for instance Ref. \cite
{barles,BMP,LMP,Lorz2,DJMR}, {and \cite{Lorz1} for a model with an additional space structure}. 
The equations studied in \cite{Lorz1}  and  \cite{Lorz2} have been also applied to address cancer heterogeneity, and the influence of the selection process on the cancer resistance to chemotherapy. \\

\noindent The novelty of our work  lies in considering a  system of two coupled equations. Difficulty of the analysis is related to the specific nonlinearities in the model, {which do not allow for component-wise estimates. The proof of boundedness of mass in the scalar equations is based on existence of sub- and supersolutions. In the case of a system, we face a difficulty which appears already in the proof of boundedness of solutions of a structure-independent model. The estimates cannot be concluded directly from the equations. To tackle this problem, we investigate the dynamics of the quotients of solutions of the two variables. Systems of equations also cause additional difficulties  when analysing the long-term dynamics in comparison to the scalar equations due to the lack of a rich class of entropies.  Convergence to a stationary positive Radon measure has been previously studied  for a scalar integro-differential equation which is linear in the nonlocal term as in Ref.  \cite{JaRa}.  This is often referred to as the Evolutionarily Stable Distribution. To deal with model nonlinearities, we make use of a Lyapunov function established previously for a finite dimensional counterpart of the model in Ref. \cite{Getto} and we show that  the total masses of solutions  tend asymptotically to the same equilibria.
 \\
}

{ \noindent{A system of two equations describing selection and mutation in a stage-structured population has been investigated in Ref. \cite{CaCu} and \cite{CC} in the context of adaptive dynamics. Analysis of that model is based on a specific structure of nonlinearities appearing only in the mortality terms. Using irreducibility of the mutation operator and the infinite dimensional version of the Perron-Frobenius Theorem, it has been shown that solutions of the model converge to a stationary distribution, which concentrates at the point of maximum fitness in the case of the frequency of mutations tending to zero. The nonlinearity in our model is related to the growth term, which requires a different approach to the analysis of the asymptotic behaviour of the model solutions. 
The difference in the structure of nonlinear feedbacks is related to a different biological definition of the described processes. While the classical juvenile-adult dynamics is based on a loop of two positive feedbacks and no self-enhacement, the model of cell differentiation involves a negative feedback and a self-enhancement of the first population. Interestingly, the two-stage structure in our model yields stabilisation of the total populations, while even in the basic juvenile-adult models, the two-stage structure may lead to multiple attractors and limit cycles; see for example Ref. \cite{ThiemeODE}.\\
}
}
{

 \noindent  The paper is organised as follows: In Section \ref{Main}, the main results are stated. Analytical results are illustrated by numerical simulations.
{ Proofs of boundedness and strict positivity of the total masses and of the exponential decay of the model solutions  outside the set corresponding to the maximal value of the self-renewal parameter are presented in Section \ref{ProofConcentration}. Section \ref{SectionMassConv} contains the proof of mass convergence to a globally stable equilibrium. Finally, the asymptotic dynamics of the model solutions is shown in Section \ref{convergence}. Additionally, in Section \ref{Measures}, we show how to extend the analysis of our model to the framework of positive Radon measures with a suitable metric. Finally, in Section \ref{discussion} we discuss biological conclusions and ideas stemming from this work. A summary of properties of the metrics used in Section \ref{convergence} is provided in the Appendix.}

\section{Main results}
\label{Main}
\noindent We consider the following system of integro-differential equations
\begin{eqnarray}\label{sys:reg}
					\frac{\partial}{\partial t} u_1(t,x) &=& \left(\frac{2a(x)}{1+K\rho_2(t)}-1\right)pu_1(t,x), \nonumber\\
					\frac{\partial}{\partial t} u_2(t,x) &=& 2\left(1-\frac{a(x)}{1+K\rho_2(t)}\right)pu_1(t,x) - d u_2(t,x), \\
					u_1(0,x) &=& u_1^0(x),\nonumber \\
					u_2(0,x)& =& u_2^0(x), \nonumber
\end{eqnarray}
where $$\rho_i(t) = \int\limits_{\Omega} u_i (t,x) \,\D x,\:\:\: \text{for} \:\:\:  i=1,2$$ and $\Omega \subset \R$ is open and bounded. \\

\noindent In the remainder of this work we make the following assumptions on the model parameters  { and initial data}.
\begin{assumption}\label{Ass}

\begin{itemize}
\item[(i)]  $a\in C(\overline \Omega)$ with { $0< a < 1$} and $\overline\Omega$ being a closure of $\Omega$.\\
\item[(ii)]$p$, $d$ and $K$ are positive constants.\\
\item[(iii)] {$u_1^0, u_2^0 \in L^1(\Omega)$ { are strictly positive a.e. with respect to the Lebesgue measure, i.e. $\int_{B }u_i^0 \D x>0$, for every set B such that ${\cal L}^1(B)>0$, $i=1,2$.}}\\
\item[(iv)] The set of maximal values of the self-renewal parameter $a$, i.e.
{ 
\begin{equation}\label{omega_a}
\Omega_a=\arg\max\limits_{x\in\overline\Omega}a(x) = \left\{\bar{x}\in\overline\Omega\left|  \bar a:= a(\bar{x}) = \max\limits_{x\in\overline\Omega} a(x)\right.\right\}
 \end{equation} }
either consists of a single point or it is a set with a positive Lebesgue measure.
\end{itemize}
\end{assumption}
\begin{remark}
The assumption  (iv) on the self-renewal fraction $a(x)$ is made to streamline the presented analysis.  If $ \Omega_a$ consists of several isolated points, then the solution is attracted by a finite dimensional subspace spanned by Dirac deltas located at the maximum points of $a$; see Fig.~\ref{Fig2}. However, in this case  the exact pattern may also depend on the shape of function $a(x)$ near its maximal points. { Since analysis of this case requires stronger assumptions on regularity of the initial data and the function $a(x)$, we consider it separately in Theorem \ref{concentration2}.}
\end{remark}
 Existence and uniqueness of  a classical solution $u = (u_1, u_2)\in C^1([0,T), L^1(\Omega)\times L^1(\Omega))$ follow by the standard theory of ordinary differential equations in Banach spaces. 
More delicate is the question of asymptotic behaviour of the solutions of system \eqref{sys:reg}. Our goal is to show that the solution $u$ 
tends asymptotically to a stationary measure, as it is observed in the numerical simulations, see Fig.~\ref{Fig1} and Fig.~\ref{Fig2}.
The phenomenon is characterised by the following Theorem.
\begin{theorem}\label{thm:concentration}
		Let Assumptions \ref{Ass} hold and let $(u_1,u_2)$ be a solution of system \eqref{sys:reg} with initial data $(u_1^0,u_2^0)$. Then,
		$u_1$ and $u_2$ converge to { stationary} measures  with supports contained in the set  
		{ $\Omega_a$ defined in expression \eqref{omega_a}}, as $t$ tends to infinity. Moreover,
		 \begin{itemize}
		\item[(i)] If $ \Omega_a$ consists of  a single point $\bar x$ { and $\bar a= \max\limits_{x\in\overline\Omega} a(x) > \frac12$}, then the solution converges to a stationary measure  (Dirac measure multiplied by a positive constant  { $(c_1,c_2)=\left(\frac{d}{p}\frac{2\bar a-1}{K},\frac{2\bar a-1}{K}\right)$}) concentrated in $\bar x$. Convergence holds in the flat metric (bounded Lipschitz distance); see Appendix for the definition and properties of the bounded Lipschitz distance. \\
		\item[(ii)] If $\Omega_a$ is a set with positive measure { and $\bar a=  \max\limits_{x\in\overline\Omega} a(x) > \frac12$}, then the solution converges to a stationary $L^1$-function, such that \newline  $\lim\limits_{t\rightarrow +\infty}u_i(t,x)= \tilde c_i u_i^0(x) \mathbb{1}_{\Omega_a}$, for $i=1,2$, where $ \mathbb{1}_{\Omega_a}$ is the characteristic function of the set $\Omega_a$, { $\tilde c_1=\frac{d}{p}\frac{(2 \bar a-1)}{Ku_1^0 |\Omega_a|}$, and $\tilde c_2=\frac{(2 \bar a-1)}{Ku_1^0 |\Omega_a|}$}. Convergence is strong in $L^1(\Omega)$.\\
		\item[(iii)] { If $\bar a = \max\limits_{x\in\overline\Omega} a(x) \leq \frac12$, then the solution converges to zero, i.e. $\lim\limits_{t\rightarrow +\infty}u_i(t,x)=0$, for $i=1,2$. Convergence is strong in $L^1(\Omega)$.}
		\end{itemize}
\end{theorem}
{\begin{remark}
 If $a(x)\leq \frac12$ for some points $x \in  \Omega$, then the solutions of the model converge point-wise to 
zero, i.e. $\lim_{t\rightarrow \infty}(u_1(t,x),u_2(t,x)) =(0,0)$ for every $x \in \Omega_-:=\{x\in \Omega\left| a(x)\leq \frac12 \right  \}$. This is a straightforward consequence of equation \eqref{sys:reg}, since $\rho_2$ is strictly positive, as shown in Lemma \ref{lem:1}, and hence $\left(\frac{2a(x)}{1+K\rho_2(t)}-1\right)<0$ for  $x \in \Omega_-$. Therefore, we are interested in evolution of the system for $x \in \Omega_+:=\Omega\setminus \Omega_-$. Subpopulations with $a(x) \leq \frac12$ may affect 
short-term dynamics of the system; however they have no influence on the asymptotic behaviour.
 \end{remark}}
Details of the proof are presented in Sections \ref{ProofConcentration}, \ref{SectionMassConv} and \ref{convergence}.  { The proof is based on the following key steps:\\

\noindent {\it Step 1.} Uniform boundedness and strict positivity of  masses  $\rho_i(t) = \int\limits_{\Omega} u_i (t,x) \,\D x$ for $i=1,2$ (Lemma \ref{lem:1})}.
\begin{lemma}\label{lem:1}
		Let Assumptions \ref{Ass} (i)-(iii) hold with { $\bar a = \max\limits_{x\in\overline\Omega} a(x) > \frac12$}
		and let $(u_1, u_2)$ be a solution of system \eqref{sys:reg}. Then, $\rho_1$ and $\rho_2$ are uniformly bounded and strictly positive, i.e. there exists a positive lower bound, uniform in time. 
\end{lemma}
Proof of this lemma is deferred to Section \ref{boundedness}.\\ 

 \noindent   {{\it Step 2.} Exponential extinction of solutions in points outside the set $\Omega_a$ (Lemma \ref{lem:decay2}).}\\
 
\noindent We start with characterising the asymptotic behaviour of the ratios of solutions taken at different $x$ points. 
\begin{lemma}\label{asympt}
Let $x_1, x_2\in\overline \Omega$ such that $a(x_1) -a(x_2) < 0$. Then, there exists a constant $M_3>0$ such that
\[\frac{u_1(t,x_1)}{u_1(t,x_2)} \leq \frac{u_{1}^0(x_1)}{u_{1}^0(x_2)}e^{p\frac{2\left(a(x_1) - a(x_2)\right)}{1+KM_3}t} \:\: \stackrel{t\rightarrow\infty}{\longrightarrow} \:\:0, \]
a.e. with respect to the Lebesgue measure.
\end{lemma}
The proof of this lemma is deferred to Section \ref{Asympt}.

\noindent Lemma \ref{asympt} yields the following result:
\begin{corollary}\label{MassDistr}
Let $x_1, x_2\in \overline\Omega$ such that $a(x_1)=a(x_2)$. Then,  $\frac{u_1(t,x_1)}{u_1(t,x_2)}$ is constant in time.
\end{corollary}
 As a consequence of Lemma \ref{asympt} we also obtain
\begin{lemma}\label{lem:decay2}
Suppose that Assumptions \ref{Ass} (i) - (iii) hold. Then,  $u(t,x)\to0$, exponentially, as $t\to\infty$ for  $x \notin  \Omega_a$ a.e. with respect to  the Lebesque measure.\end{lemma}

\noindent The corresponding proof is presented in Section \ref{Asympt}.\\

 \noindent { {\it Step 3.} Convergence of solutions to stationary measures.}\\

\noindent Convergence to the stationary solutions follows from the property of the total masses of the solutions $(\int_{\Omega}u_1(t,x) \D x,\int_{\Omega}u_2(t,x) \D x)$. 
We show that if { $\bar a = \max\limits_{x\in\overline\Omega} a(x) > \frac12$}, then the solutions converge to the stationary state of the system with $\bar a=\max\limits_{x\in\overline\Omega}a(x)$. 

\begin{theorem}\label{MassConv}
Suppose that Assumptions \ref{Ass} hold, { $\bar a = \max\limits_{x\in\overline\Omega} a(x) > \frac12$} and  \\ $(\rho_1,\rho_2)= (\int_{\Omega}u_1(\cdot,x)\D x,\int_{\Omega}u_2(\cdot,x) \D x)$ be total masses of solutions of \eqref{sys:reg}. It holds that $(\rho_1(t),\rho_2(t))\rightarrow (\bar\rho_1,\bar\rho_2),$  as $t\rightarrow \infty$,
where $(\bar \rho_1,\bar \rho_2)$ are { stationary } solutions of the corresponding ordinary differential equations model with the maximal value of the self-renewal parameter $\bar a$,  { i.e.,
\begin{eqnarray}\label{sys:2eqmaxa}
					0 &=& \left(\frac{2\bar a}{1+K \bar \rho_2}-1\right)p \bar \rho_1, \nonumber\\
					0 &=& 2\left(1-\frac{\bar a}{1+K \bar \rho_2}\right)p \bar \rho_1 - d \bar \rho_2.
\end{eqnarray}}
\end{theorem}
{
Direct calculations based on equations \eqref{sys:2eqmaxa} yield
\begin{corollary}\label{Masses}
Total masses converge to the values $\bar \rho_1=\frac{d}{p}\frac{2\bar a-1}{K}$ and $\bar \rho_2=\frac{2\bar a-1}{K}$.
\end{corollary}
}
{ Details of the proof of mass convergence are deferred to Section \ref{SectionMassConv}. \\

\noindent If  $\Omega_a$ consists of a single point $\bar x$ and { $\bar a= \max\limits_{x\in\overline\Omega} a(x) > \frac12$}, then the exponential decay of the solutions outside the set $\Omega_a$  together with the convergence of total masses, yields convergence of the solutions to a stationary measure concentrated at $\bar x$ (a Dirac measure multiplied by a positive constant). In the case of $\Omega_a$ having a positive Lebesgue measure, convergence of solutions  together with Corollary \ref{MassDistr} on the stationary distribution of masses among different domain points yields convergence of solutions to the stationary equilibrium.  Further details of the proof of convergence of solutions to the stationary measures are given in Section \ref{convergence}.
\begin{remark} In the case $\Omega_a=\{ \bar x\}$,  the convergence holds in the weak$*$ topology of Radon measures. In general, we cannot expect the strong (norm-{ total variation}) convergence of the solution to a stationary solution. If the set $\Omega_a \subset \R$
has zero Lebesgue measure and consists of a single point (compare Assumptions  \ref{Ass} (iv)), then the model solutions for any finite time point  are uniformly continuous with respect to the Lebesgue measure and $u_i(t,\cdot) \mathcal L^1 \rightarrow c_i \delta_{\bar x}$, weakly$*$, for $i=1,2$.  Here,  $u_i(t,\cdot)\mathcal L^1$ denotes the measure such that $u$ is its Radon-Nikodym derivative with respect to $\mathcal L^1$. 

Hence, the distance between the two solutions  $TV(u_i(t,\cdot),c_i\delta_{\bar{x}})\ge 2c_i$.  
The problem can be solved by considering convergence with respect to a suitable metric, for example the  flat metric (bounded Lipschitz distance); for details see Section 5.
\end{remark}}
\begin{figure}[htpb] 
\begin{center}
\includegraphics[width=0.45 \textwidth]{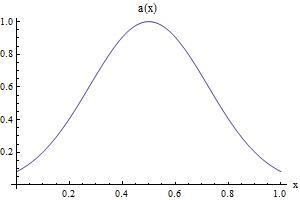} \\
\vspace{1cm}
\includegraphics[width=0.5 \textwidth]{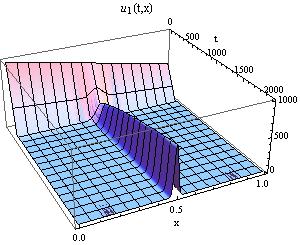} \\
\vspace{1cm}
\includegraphics[width=0.5 \textwidth]{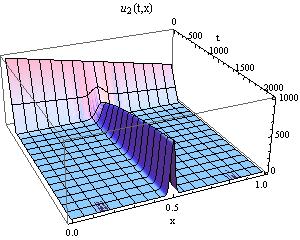}  
\caption{Numerical simulations of the model \eqref{sys:reg} with the self-renewal function $a(x)$ having a single local maximum (shown in the upper panel). Parameters used in  the simulation: $K=0.01$, $p=1$, $d=0.2$ and the initial data: $u_1^0(x) = 1000-500x$,
$u_2^0(x) = 1000x^2$. We observe mass concentration in the point $\bar x =\arg  \max\limits_{x\in \Omega}a(x)$ and convergence of the mass to a stable stationary value.} \label{Fig1}
\end{center}
\end{figure}

{ If the support of { $\bar a$ is not a single point set, then the stationary distribution of masses depends on the initial conditions. If $\Omega_a$ has a positive Lebesgue measure, then the distribution of masses results from Corollary \ref{MassDistr}. If $\Omega_a$ consists of a discrete set of points, then the stationary solution takes the form of a linear combination of Dirac deltas; see Fig \ref{Fig2}. We show that in such case the limit function depends on the shape of $a(x)$ in the neighbourhood of the concentration points.

\begin{theorem}[{\bf Co-existence of different stationary solutions}]\label{concentration2}
Let Assumptions \ref{Ass} (i)-(iii) hold and, additionally, the initial functions $u_1^0, u_2^0 \in C(\Omega)$. Let the set $\Omega_a$ of the maximum values of the self-renewal parameter $a$ (as defined in expression \eqref{omega_a}) consist of two points $\Omega_a=\{\bar x_1, \bar x_2\}$ and $u_1^0$ be strictly positive on $\Omega_a$. Then,
		\begin{itemize}
		\item[(i)] If there exists a diffeomorphism $\Phi \in C^1(U_1)$, where $U_1$ is an open neighbourhood of $\bar x_1$, such that
		\begin{eqnarray}\label{PhiProp}
		\Phi(\bar x_1)&=&\bar x_2,\nonumber\\
		a(x)&=&a(\Phi(x))\quad \text{for all} \quad x\in U_1,
		\end{eqnarray}
		 then solutions $(u_1,u_2)$ of system \eqref{sys:reg} converge to stationary measures, which are linear combinations of Dirac measures concentrated in $\bar x_1$ and $\bar x_2$, multiplied by strictly positive constants. 
		\item[(ii)]  If the mapping $\Phi$ with the properties defined by condition \eqref{PhiProp} is only a homeomorphism with a singular Jacobian of the inverse mapping $\Phi^{-1}$ at $\bar x_2$, then solutions $(u_1,u_2)$ of system \eqref{sys:reg} converge to stationary measures concentrated in $\bar x_2$.
		\end{itemize}
\end{theorem}

The proof of this theorem is deferred to Section \ref{convergence}.
\begin{remark}
If $a$ is an analytic function and $\Omega \subset \R$, then a diffeomorphism satisfying condition \eqref{PhiProp} exists if the first nonconstant nonzero terms of Taylor expansion of the function $a(x)$  are of the same order. 

This observation suggests how to construct $a(x)$ with $\Omega_a=\{\bar x_1, \bar x_2\}$ such that solutions extinct at one of the points of $\Omega_a$. For example, we may define $a(x)$ with $x\in \Omega=[0,1]$ such that 
\begin{equation*}
a(x):=\left\{\begin{array}{ccc}
- (x-\frac{1}{4})^2 + \frac{9}{10} &for& x\in [0, \frac{3}{8}),\\
- (x-\frac{3}{4})^4+\frac{9}{10} &for&  x\in (\frac{5}{8}, 1].
\end{array}\right.
\end{equation*}
and a smooth extension of $a(x)$ on the interval $(\frac{3}{8},\frac{5}{8})$ satisfying $0<a(x)<1$. We obtain $\Omega_a=\{\frac{1}{4},\frac{3}{4}\}$, and a mapping $\Phi(x)=\sqrt{x-\frac{1}{4}}+\frac{3}{4}$ satisfying condition \eqref{PhiProp} on $U_1=(\frac{1}{4}-\varepsilon,\frac{1}{4}+\varepsilon)$, where $\varepsilon<\frac{1}{8}$. Consequently, $\Phi^{-1}(x)= (x-\frac{3}{4})^2+\frac{1}{4}$ and it is singular at $x=\frac{3}{4}$.
Hence, the total mass concentrates at the point $x=\frac{3}{4}$ and there is an extinction of mass at $x=\frac{1}{4}$.
\end{remark}
}

\begin{figure}[htpb] 
\begin{center}
\includegraphics[width=0.45 \textwidth]{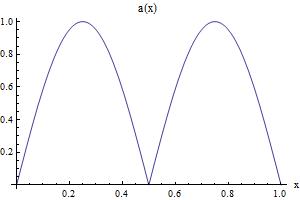} \\
\vspace{1cm}
\includegraphics[width=0.5 \textwidth]{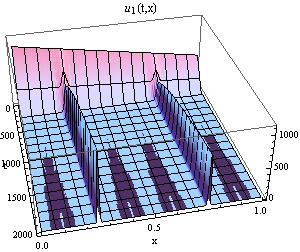} \\
\vspace{1cm}
\includegraphics[width=0.5 \textwidth]{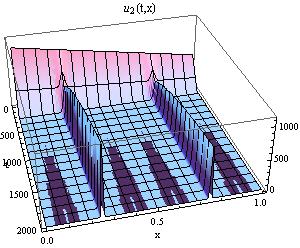}  
\caption{Numerical simulations of the model \eqref{sys:reg} with the self-renewal function $a(x)$ having two equal local maxima (shown in the upper panel) and the parameters the same as in Fig. \ref{Fig1}. We observe mass concentration in two points corresponding to the maximum of the function $a(x)$ with unequal distribution of the mass between the two points.} \label{Fig2}
\end{center}
\end{figure}
}{

\section{Proof of mass concentration}\label{ProofConcentration}
\subsection{Boundedness and strict positivity of masses}\label{boundedness}
All considerations in this Section hold for $x\in \Omega$ a.e. with respect to the Lebesque measure.\\

\noindent { First, we notice that the solutions $(u_1,u_2)$ are nonnegative, since $a(x)/(1+K\rho_2)<1$. 
}{ Before proving Lemma \ref{lem:1}, we show the following technical result.
\begin{lemma}\label{Bound}
Under the assumptions of Lemma~\ref{lem:1}, the function $U = \frac{u_1}{u_2}$ is uniformly bounded on $\Omega \times \R^+$.
\end{lemma}
\begin{proof}
The equation for $U(t,x) = \frac{u_1(t,x)}{u_2(t,x)}$ reads for $t>0$
\begin{eqnarray}\label{eq_u}
		\frac{\partial}{\partial t} U(t,x) &=&  U(t,x)\Biggr(p\left(\frac{2a(x)}{1+K\rho_2(t)}-1\right)+d  	\\
		& & -2p\left(1 - \frac{a(x)}{1+K\rho_2(t)}\right)U(t,x)\Biggr) \nonumber.
\end{eqnarray}
Since $$p\left(\frac{2a(x)}{1+K\rho_2(t)}-1\right)+d \leq 2p\bar a +d$$ and $$1 - \frac{a(x)}{1+K\rho_2(t)}>1-\bar a,$$ 
and the right-hand side of equation \eqref{eq_u} is a logistic type nonlinearity, we conclude that
\[U(t,x) \leq \max\left\{U(0,x),\frac{2p\bar{a} +d}{2p(1-\bar{a})}\right\} =: M_1 \quad\forall \; (t,x) \in [0,T)\times\Omega.\]
By definition of $U$, we can infer that
\[u_1(t,x) \leq M_1 u_2(t,x) \quad \forall\; (t,x)\in[0,T)\times\Omega. \]
\end{proof}
 As a straightforward consequence of Lemma \ref{Bound}, we deduce 
\begin{corollary}
Under the assumptions of Lemma~\ref{lem:1}, it holds
\begin{eqnarray}
		\int\limits_{\Omega} u_1(t,x)\,\D x \leq M_1 \int\limits_{\Omega} u_2(t,x)\,\D x = M_1 \rho_2(t). \label{ineq:u12}
\end{eqnarray}
\end{corollary}
\noindent Now we state another technical result in the spirit of Lemma \ref{Bound}.
\begin{lemma}\label{Bound2} 
There exist constants $M_4>0$  and $0<\gamma<1$ such that $\rho_2(t) \leq M_4\rho_1^{\gamma}(t)$ for all $t\geq 0$. 
\end{lemma}
\begin{proof}
Calculating the derivative of the quotient of  $\rho_2(t)$ and $\rho_1^\gamma(t)$, we obtain\\
\begin{eqnarray*}
		\frac{d}{d t} \frac{\rho_2(t)}{\rho_1^{\gamma}(t)} &=& 	\frac{\frac{d}{d t} \rho_2(t) \rho_1^{\gamma}(t) 
		 - 	\rho_2(t) \gamma \rho_1^{\gamma-1}(t)\frac{d}{dt }\rho_1(t)}			{\rho_1^
{2\gamma}(t)} \\
		&=& \frac{\int_{\Omega}\left( 2(1-\frac{a(x)}				{1+K\rho_2(t)})pu_1(t,x) -d u_2(t,x)\right)\,\D x}					{\rho_1^{\gamma}(t)} \\
		& & - \frac{\rho_2(t)}						{\rho_1^
{\gamma}(t)}\frac{\gamma\int_{\Omega} 				\left(\frac{2a(x)}{1+K\rho_2(t)} - 1\right)pu_1(t,x)\,			\D x}{\rho_1} \\
		&\leq& \frac{\int_{\Omega}\left( 2 (1- \frac{a(x)}{1+K\rho_2(t)})pu_1(t,x) -d u_2(t,x) \right)\,\D x}					{\rho_1^{\gamma}(t)} + \frac{\rho_2(t)}						{\rho_1^
{\gamma}(t)}\gamma p \\
		&\leq& 2p\rho_1^{1-\gamma}(t) + \frac{\rho_2(t)}				{\rho_1^{\gamma}(t)}(\gamma p-d) \leq 					2pM_2^{1-\gamma} + \frac{\rho_2(t)}						
{\rho_1^{\gamma}(t)}(\gamma p -d).
\end{eqnarray*} 

\noindent This estimate holds for arbitrary $\gamma \in (0,1)$, so in particular for those satisfying $\gamma p -d < 0$. Arguing as before, we deduce that, for all $t\geq 0$,
\begin{equation}\label{rhoestimate}
\frac{\rho_2}{\rho_1^{\gamma}}(t) \leq \max\left\{\frac{\rho_2(0)}{\rho_1^{\gamma}(0)}, \frac{2pM_2^{1-\gamma}}{d - \gamma p}\right\} =: M_4.
\end{equation}
\end{proof}
\vspace{0.5cm}
Equipped with Lemma \ref{Bound} and Lemma \ref{Bound2}, we prove Lemma \ref{lem:1}.

\begin{proof}[of Lemma \ref{lem:1}]
{\\}

\noindent (i) First, we show uniform boundedness of masses $\rho_1$ and $\rho_2$, {which yields also the global existence of solutions   $(u_1, u_2)\in C^1([0,\infty), L^1(\Omega)\times L^1(\Omega))$}.\\ 

\noindent To show boundedness of $\rho_1$, we apply inequality  \eqref{ineq:u12} to the first equation of system \eqref{sys:reg}
\begin{eqnarray*}
	\frac{\partial}{\partial t} u_1(t,x) &=& \left(\frac{2a(x)}{1+K\rho_2(t)}-1\right)pu_1(t,x) \leq \left(\frac{2a(x)}{1+ \frac{K}{M_1}\rho_1(t)}-1\right)pu_1(t,x) \\
	&\leq& \left(\frac{2\bar{a}}{1+\frac{K}{M_1}\rho_1(t)}-1\right) pu_1(t,x).
\end{eqnarray*}
Integrating this inequality over $\Omega$ yields
\begin{eqnarray}\label{ineq:u1}
		\frac{d}{d t} \rho_1(t) \leq \left(\frac{2\bar{a}}{1+\frac{K}{M_1}\rho_1(t)}-1\right)p\rho_1(t).
\end{eqnarray}
Using a similar argument as in the proof of Lemma \ref{Bound}, we conclude that%
\begin{eqnarray}\label{ineq:rho1}
\rho_1(t) \leq \max\left\{\rho_1(0),\frac{(2\bar{a}-1)M_1}{K}\right\} =: M_2.
\end{eqnarray}
Boundedness of $\rho_2$ results from the second equation of system \eqref{sys:reg}, nonnegativity of $\rho_2$ and the assumptions on $a$.
It holds
\begin{eqnarray*}
		\frac{\partial}{\partial t} u_2(t,x) &=& 2\left(1-\frac{a(x)}							{1+K\rho_2(t)}\right)pu_1(t,x) -d u_2(t,x) \leq 						2pu_1(t,x) -d u_2(t,x).
\end{eqnarray*} 
Integrating over $\Omega$ and using \eqref{ineq:rho1}, we obtain
\[\frac{d}{d t} \rho_2(t) \leq 2p\rho_1(t) -d\rho_2(t) \leq 2pM_2 -d\rho_2(t). \]
Hence, we conclude that 
\begin{eqnarray}\label{ineq:rho2}
\rho_2(t) \leq \max\left\{\rho_2(0),\frac{2pM_2}{d}\right\} =: M_3.
\end{eqnarray}

\noindent (ii) We show that masses  $\rho_1$ and $\rho_2$ have a strictly positive lower bound, uniform in time.

{ 

We estimate the growth of $\rho_1$ using a decomposition of the domain $\Omega=\Omega_- + \Omega_+$, where $\Omega_-:=\{x\in \Omega\left| a(x)\leq \frac12 \right  \}$ and $\Omega_+:=\{x\in \Omega\left| a(x)> \frac12 \right  \}$.\\
 
 First, we assume that the set $\Omega_-$ is nonempty, i.e. $\int_{\Omega_-} u_1^0(x)>0$.
 We denote $$\rho_1^-(t)=\int\limits_{\Omega_-} u_1(t,x)\; \D x \quad \text{and}  \quad \rho_1^+(t)=\int\limits_{\Omega_+} u_1(t,x)\; \D x.$$
Using the explicit form of the solution
\begin{equation}\label{uexplicit}
u_1(t,x)=u^0_1(x)e^{\int \limits_0^t\left(\frac{2a(x)}{1+K\rho_2(\tau)}-1\right)p \;\D \tau}
\end{equation}
 and the properties of the function $a(x)$ on the two subdomains, we obtain
 \begin{eqnarray}\label{rho1quotient}
	\frac{\rho_1^+(t)}{\rho_1^-(t)}  &=& \frac{\int_{\Omega_+} u_1^0(x) e^{\int_0^t \left(\frac{2a(x)}{1+K\rho_2(\tau)} -1\right)p\;\D \tau}\,\D x}{\int_{\Omega_-} u_1^0(x) e^{\int_0^t \left(\frac{2a(x)}{1+K\rho_2(\tau)} -1\right)p\; \D \tau}\,\D x}
	\geq \frac{ \inf_{\Omega_+} e^{\int_0^t \left(\frac{2a(x)}{1+K\rho_2(\tau)} -1\right)p\;\D \tau}\,\int_{\Omega_+} u_1^0(x) \;\D x}{ \sup_{\Omega_-} e^{\int_0^t \left(\frac{2a(x)}{1+K\rho_2(\tau)} -1\right)p\; \D \tau}\,  \int_{\Omega_-} u_1^0(x)\;\D x}\nonumber\\
&=& \frac{ e^{\int_0^t \left(\frac{1}{1+K\rho_2(\tau)} -1\right)p\;\D \tau}\,\int_{\Omega_+} u_1^0(x) \;\D x}{ e^{\int_0^t \left(\frac{1}{1+K\rho_2(\tau)} -1\right)p\; \D \tau}\,  \int_{\Omega_-} u_1^0(x)\;\D x}=\frac{\rho_1^+(0)}{\rho_1^-(0)}.
\end{eqnarray}
Combining estimates \eqref{rho1quotient} and \eqref{rhoestimate} yields
\begin{equation}\label{rhoestimate2}
\rho_2(t) \leq M_4(\rho_1^+(t) + \rho_1^-(t))^{\gamma}\leq M_4 \left( \rho_1^+(t)\left(1+ \frac{\rho_1^+(0)}{\rho_1^-(0)} \right)\right)^{\gamma}=M_5 \left(  \rho_1^+(t)\right)^{\gamma}
\end{equation}
with $M_5=M_4 \left(1+ \frac{\rho_1^+(0)}{\rho_1^-(0)} \right)^{\gamma}$.

With estimate \eqref{rhoestimate2} at hand, we show that $\rho_1^+$ is strictly positive for every $t\in \R^+$. We estimate its dynamics
\begin{eqnarray*}
		\frac{d}{d t} \rho_1^+(t) &=& \int\limits_{\Omega_+} \left(\frac{2a(x)}{1+K\rho_2(t)} -1\right)pu_1(t,x)\,\D x\\ 
		&\geq& \left(\frac{2\underline{a}}{1+KM_5\left(\rho_1(t)^+\right)^{\gamma}} -1\right)p\rho_1(t),
\end{eqnarray*}
where  $\underline{a}= \min\limits_{x\in\overline \Omega_+}a(x)>\frac12$.

 The term in the brackets is strictly positive for $\rho_1^+$ small enough, i.e. for
\[\rho_1^+(t) \leq \left( \frac{2\underline{a} -1}{KM_5}\right)^{\frac{1}{\gamma}},\]
which is a positive constant,  since  $\underline{a}>\frac{1}{2}$.

Hence, we obtain the estimate
\[\rho_1(t) \geq \min\Big\{\rho_1(0),  \left( \frac{2\underline{a} -1}{KM_5}\right)^{\frac{1}{\gamma}}\Big\} =: M_6   \quad\forall \; t \in [0,\infty). \]
}

\noindent  Consequently, we obtain the strict positivity of $\rho_1$ and using the second equation of \eqref{sys:reg},  also the strict positivity of $\rho_2$.
{ In the case of $\Omega_-=\emptyset$, it holds $\rho_1=\rho_1^+$ and the proof is complete if we set   $M_5=M_4$.
}
\end{proof}
\subsection{Asymptotic behaviour of  the solutions}\label{Asympt}
In the next step,  we show that the first component of  the solution of system \eqref{sys:reg} tends to zero for $\bar x \notin\Omega_a$ a.e. with respect to the Lebesque measure.

\begin{proof}[of Lemma \ref{asympt}]
We choose two points $x_1, x_2\in\overline \Omega$ such that $a(x_1) -a(x_2) < 0$, and calculate%
\begin{eqnarray*}
		\frac{\partial}{\partial t} \frac{u_1(t,x_1)}{u_1(t,x_2)} 	
		= p\frac{u_1(t,x_1)}{u_1(t,x_2)}\left(2\frac{a(x_1)-a(x_2)}						{1+K\rho_2(t)}\right) \leq p\frac{u_1(t,x_1)}									{u_1(t,x_2)}\left(2\frac{a(x_1)-a(x_2)}{1+KM_3}\right).
\end{eqnarray*}
Solving the above differential inequality for $\frac{u_1(t,x_1)}{u_1(t,x_2)}$, we obtain the assertion of this Lemma by the choice of $x_1$ and $x_2$. 
\end{proof}

\begin{lemma}\label{asympt2}
Let $x_1,x_2 \in\overline \Omega$ be such that $a(x_1) - a(x_2) < 0$, then 
\[\frac{u_2(t,x_1)}{u_2(t,x_2)} \stackrel{t\rightarrow\infty}{\longrightarrow} 0,\]
a.e. with respect to the Lebesque measure.
\end{lemma}
\begin{proof}
		We use a similar ansatz as in Lemma \ref{asympt} and 				calculate for $t>0$ \\
		\begin{eqnarray*}
				\frac{\partial}{\partial t} \frac{u_2(t,x_1)}{u_2(t,x_2)} &=& 						2\left(1-\frac{a(x_1)}{1+K\rho_2(t)}\right) 					p\frac{u_1(t,x_1)}{u_2(t,x_2)} \\
				& & - 				
				2\left(1-\frac{a(x_2)}{1+K\rho_2(t)}\right)p 				\frac{u_2(t,x_1)}{u_2(t,x_2)}\frac{u_1(t,x_2)}					{u_2(t,x_2)}.
		\end{eqnarray*}
		Applying Lemma \ref{asympt}, we obtain 
		\begin{eqnarray*}
		\frac{\partial}{\partial t} \frac{u_2(t,x_1)}{u_2(t,x_2)} &=& 					p\frac{u_1(t,x_2)}											{u_2(t,x_2)}\Bigg(2\left(1-\frac{a(x_1)}				
		{1+K\rho_2}\right)\frac{u_1^0(x_1)}							{u_1^0(x_2)}e^{\frac{2\left(a(x_1)-a(x_2)\right)t}						{1+KM_3}} \\
		& & - 2\left(1-\frac{a(x_2)}								{1+K\rho_2}\right)\frac{u_2(t,x_1)}							{u_2(t,x_2)}\Bigg).
		\end{eqnarray*}
		Thus, we deduce the following bound for 					$\frac{u_2(t,x_1)}{u_2(t,x_2)}$\\
	{	\[\frac{u_2(t,x_1)}{u_2(t,x_2)} \leq 							  \frac{\left(1-\frac{a(x_1)}								  {1+KM_3}\right)\frac{u_1^0(x_1)}				
		  {u_1^0(x_2)}e^{\frac{2(a(x_1) - a(x_2))t}					  {1+KM_3}}}{1- a(x_2)},\]}\\
		where the right hand side tends exponentially to zero, as $t$ 			tends to infinity. 
		
	\noindent	This concludes the proof. 
\end{proof}
 Having shown the dynamics of the ratios of the values of a solution at different $x$ points, we prove that the solutions converge to zero outside the set of points with a maximum value of the parameter $a(x)$.
\begin{proof}[of Lemma \ref{lem:decay2}]
Let $\tilde x$ be a point different from $\bar x$ and assume that \\$\lim_{t \rightarrow \infty} u(t,\tilde x) > 0$. Continuity of $a(x)$ implies that the set of $x$, such that $a(x)>a(\tilde x)$, is an open nonempty 
set and, therefore, it has positive measure. Since Lemma \ref{asympt} holds for every $x, \tilde x \in\overline \Omega$ such that $a(x) - a(\tilde x) > 0$, we conclude that $u(t,x)$ tends exponentially to $+ \infty$  
for every $x$ such that $a(x)>a(\tilde x)$. This is, however, in contradiction with the uniform boundedness of the mass $\int_{\Omega}u(t,x)\D x$.
\end{proof}

\section{Proof of convergence of the total mass}\label{SectionMassConv}

We begin the proof of Theorem  \ref{MassConv} by showing the following lemma, which allows comparing two dynamical systems. \begin{lemma}\label{lem:dynamics}
{ Let $t\rightarrow X_F (t,\cdot)$ be a one-parameter family of $C^1$-diffeomorphisms (semiflows) $X_F (t, (0,\infty)\times(0,\infty))\subset  (0,\infty)\times(0,\infty)$, for every $t\geq 0$, generated by the ordinary differential equation
\begin{equation}\label{ODEproblem}
\frac{du}{dt}=F(u)
\end{equation} 
such that $V\in C^1((0,\infty)\times(0,\infty))$, with a single minimum $\bar u$, is a strict Lyapunov functional, i.e.  $\frac{d}{dt}X_F(t,u)|_{t=0}\cdot \nabla V(u) = 0$ for $u=\bar u$ and $\frac{d}{dt}X_F(t,u)|_{t=0}\cdot \nabla V(u) < 0$ otherwise.  Then, if $\tilde u$ is a solution of 
\begin{equation}\label{perturbedproblem}
\frac{d \tilde u}{dt}=F(\tilde u) + f,
\end{equation}
 where $lim_{t\rightarrow\infty}sup_{\tau\in[t,\infty)}|f(\tau)|=0$ and $\overline {\rm Im (\tilde u(\cdot))}:=\overline{\cup_{t\in [0,\infty)} \{\tilde u(t)\}} \subset (0,\infty)\times (0,\infty)$ is compact, then  $\tilde u(t) \rightarrow \bar u$ for $t \rightarrow \infty$.}
\end{lemma} 
\begin{proof}

\noindent For arbitrary $a>\bar V$, we define a truncation\\
\begin{equation*}
V_a(u):=\left\{\begin{array}{ccc}
V(u)-a &\text{ if }& V(u)\geq a,\\
0 &\text{ if }& V(u)<a.
\end{array}\right.
\end{equation*} \\
Since $V_a \in W^{1,\infty}(U)$, where $U$ is the intersection of all convex sets containing $\overline{\rm Im(\tilde u(\cdot))}$,  {$U={\rm conv}(\overline{\rm Im(\tilde u(\cdot))})\subset (0,\infty)\times(0,\infty)$}, and $\frac{d}{d t} \tilde u \in L^1(\Omega)$, then we can define  the time derivative of $V_a(\tilde u(t))$ using the chain rule. 
$\nabla_{u} V_a$ is defined in a classical sense only outside the set $V( u)=a$, but it has a Clarke derivative, i.e. a generalised subdifferential for a locally Lipschitz function  \cite{Clarke}, on the set $V=a$.  In the following, $\nabla_{\tilde u} V_a(\tilde u)$  is an extension of the classical definition, involving the maximal element of
the Clarke derivative, to the set where the classical derivative is not defined.  \\

\noindent Let us define $\beta: \overline{{\rm Im} V(u)}\rightarrow (0,\infty)$ such that
\begin{equation*}
\beta(x)=\inf\limits_{\{u\in U| V_a(u)=x\}}\left\{\frac{d}{dt}X_F(t,u)|_{t=0}\cdot V(u) \right\}.
\end{equation*}
Since $\beta$ is a continuous function defined on a compact set, it achieves a strictly positive minimum. Furthermore, for the truncation function $V_a$, there exists a positive constant $\tilde \beta_a$ such that $\beta(V_a)\geq  \tilde \beta_a V_a$.
Hence, we obtain\\
\begin{equation}\label{dV_a}
\frac{dV_a(\tilde u(t))}{dt}\leq - \tilde \beta_a V_a(\tilde u(t))+ \nabla_{\tilde u} V_a(\tilde u(t)) \cdot f(t).
\end{equation} 
Using compactness of the set $U$, we estimate $\nabla_{\tilde u} V_a(\tilde u(t)) $ by its $L^{\infty}$ norm, which yields the following inequality, 
\begin{equation*}\label{dV_a2}
\frac{dV_a(\tilde u(t))}{dt}\leq - \tilde \beta_a V_a(\tilde u(t))+ C |f(t)|,
\end{equation*}
where  $C = \left\|\nabla_{\tilde{u}}V\right\|_{L^{\infty}(U)}$.

\noindent Integrating the above estimate, we obtain 
\begin{equation}\label{Vaest}
V_a(\tilde u(t))\leq V_a(u_0)e^{-\tilde \beta_a t} + \int_0^t  |f(\tau)| e^{-\tilde \beta_a (t-\tau)}\D\tau.
\end{equation}
 We show that the right-hand side of inequality \eqref{Vaest} tends to zero for $t \rightarrow \infty$.
\begin{eqnarray*}
&&\int_0^t  |f(\tau)| e^{-\tilde \beta_a (t-\tau)}\D\tau = \int_0^{\frac{t}{2}}  |f(\tau)| e^{-\tilde \beta_a (t-\tau)}\D\tau +\int_{\frac{t}{2}}^t  |f(\tau)| e^{-\tilde \beta_a (t-\tau)}\D\tau \\
&&\leq \sup_{\tau\in \R^+}|f(\tau)| \int_0^{\frac{t}{2}} e^{-\tilde \beta_a (t-\tau)}\D\tau +  \sup_{\tau\in [\frac{t}{2},\infty]}|f(\tau)| \;\int_{\frac{t}{2}}^t  e^{-\tilde \beta_a (t-\tau)}\D\tau\\
&& \leq \sup_{\tau\in \R^+}|f(\tau)| \; \frac{1}{\tilde \beta_a}\;e^{-\frac{\tilde \beta_a t}{2}} \left( 1- e^{-\frac{\tilde \beta_a t}{2}}\right) +  \sup_{\tau\in [\frac{t}{2},\infty]}|f(\tau)| \;  \frac{1}{\tilde \beta_a}\; \left( 1- e^{-\frac{\tilde \beta_a t}{2}}\right). 
\end{eqnarray*}
Since, by assumption $\lim_{t \rightarrow \infty} \sup_{\tau\in [\frac{t}{2},\infty]}|f(\tau)|=0$, passing to the limit, we obtain $$\lim_{t \rightarrow \infty} \int_0^t  |f(\tau)| e^{-\tilde \beta_a (t-\tau)}\D\tau =0.$$
Convergence holds for every $a$, which yields convergence  $V(\tilde u(t))\rightarrow \bar V$, i.e. to
 the minimum of the function $V$. In turn, this ensures that $\tilde u(t)\rightarrow \bar u$. 
\end{proof}
}

\begin{proof}[of Theorem \ref{MassConv}] 
To apply Lemma \ref{lem:dynamics} to system \eqref{sys:reg}, we consider a finite dimensional model obtained by setting $a(x)$ to a constant value $\bar a$
\begin{eqnarray}\label{sys:2eq}
					\frac{d}{d t} v_1 &=& \left(\frac{2\bar a}{1+K v_2}-1\right)p v_1, \nonumber\\
					\frac{d}{d t} v_2 &=& 2\left(1-\frac{\bar a}{1+K v_2}\right)p v_1 - d v_2, \\
					v_1(0) &=& v_1^0,\nonumber \\
					v_2(0) &=& v_2^0. 	\nonumber
					\end{eqnarray}
{ Note that the above equation generates a $C^1$-semiflow, which is invariant on $(0,\infty) \times (0,\infty).$ }	
\noindent  We check that the two systems \eqref{sys:2eq} and \eqref{sys:mass2} fulfill the assumptions of Lemma \ref{lem:dynamics}. 
				
Lyapunov function for system \eqref{sys:2eq} has been previously constructed in Ref. \cite{Getto}. It assumes the form
\begin{equation}
V(v_1,v_2):=\frac{1}{pG(\bar v_2)}V_{1}(v_1)+\frac{1}{d}V_{2}(v_2),\label{Lyapunov}
\end{equation}
where
\begin{eqnarray*}
V_{1}(v_1) & :=&\frac{v_1}{\bar v_{1}}-1-\ln\frac{v_1}{\bar v_1},\\
V_{2}(v_2) & :=&\frac{v_2}{\bar v_2}-1-\frac{1}{\bar v_2}\int_{\bar v_2}^{v_2}\frac{G(\bar v_2)}{G(\xi)}d\xi,
\end{eqnarray*}
$(\bar v_{1},\bar v_{2})$ is the stationary solution, and
\begin{equation}
G(v_2):=2\left(1-\frac{\bar a}{1+kv_2}\right)\text{ for }v_2\geq0.\label{G}
\end{equation}

\noindent { Lyapunov function \eqref{Lyapunov} is well-defined for every $(v_1,v_2)\in (0,\infty)\times (0,\infty)$. Moreover, $V\in C^{\infty}(0,\infty)\times (0,\infty)$.}
{ \noindent Note that $V_1(v_1)$ is strictly convex and therefore $\frac{\partial}{\partial v_1} V_1 \not= 0$ for $v_1\not= \bar v_1$. Similar observation holds for $V_2(v_2)$. Hence $(\bar v_1,\bar v_2)$ is the global minimum of the Lyapunov function.}  \\

\noindent Direct calculations, as provided in \cite{Getto}, allow to check that
\begin{equation}
\frac{d}{dt}V(v_1(t),v_2(t))\leq0,\label{eq:non_pos2}
\end{equation}
for the solutions of system \eqref{sys:2eq}. { Moreover, the equality $\frac{d}{dt}V(v_1(t),v_2(t))=0$ holds only for the stationary solution $(\bar v_1, \bar v_2)$}.\\

\noindent To show convergence of the total mass of the solution of system \eqref{sys:reg} to a global equilibrium, we integrate equations \eqref{sys:reg} with respect to $x$ and obtain\\
\begin{eqnarray}\label{sys:mass}
					\frac{d}{d t} \rho_1(t) &=& \int_{\Omega}\left(\frac{2a(x)}{1+K \rho_2(t)}-1\right)p u_1(t,x)\D x, \nonumber\\
					\frac{d}{d t} \rho_2(t) &=& 2\int_{\Omega} \left(1-\frac{a(x)}{1+K \rho_2(t)}\right)p u_1(t,x)\D x-  d \int_{\Omega} u_2(t,x) \D x, \\
					\rho_1(0) &=& \int_{\Omega} u_1^0(x) \D x, \nonumber\\
					\rho_2(0) &=& \int_{\Omega}u_2^0(x) \D x.\nonumber
\end{eqnarray}
This can be rewritten as
\begin{eqnarray}\label{sys:mass2}
					\frac{d}{d t} \rho_1(t) &=& \left(\frac{2 \bar a}{1+K \rho_2(t)}-1\right)p \rho_1(t) + \frac{2p}{1+K\rho_2(t)} \int_{\Omega}\left(a(x)-\bar a\right)u_1(t,x) \D x,\nonumber \\
					\frac{d}{d t} \rho_2(t) &=& 2 \left(1-\frac{\bar a}{1+K \rho_2(t)}\right)p \rho_1(t)\nonumber\\ & & + \frac{2p}{1+K \rho_2(t)}\int_{\Omega} \left(\bar a - a(x)\right) u_1(t,x)\D x -d\rho_2(t), \\
					\rho_1(0) &=& \int_{\Omega} u_1^0(x) \D x,\nonumber \\
					\rho_2(0) &=& \int_{\Omega}u_2^0(x) \D x.\nonumber
\end{eqnarray}

\noindent { By Lemma \ref{lem:1}, $\overline {\rm Im ((\rho_1(\cdot),\rho_1(\cdot))}\subset (0,\infty)\times (0,\infty)$ and it  is compact.
 \\
 
\noindent To show that the perturbation function on the right-hand side converges to zero as $t\rightarrow \infty$, we calculate
\begin{eqnarray*}
\int_{\Omega}\left(a(x)-\tilde a\right)u_1(t,x) \D x = \int_{\Omega_a}\left(a(x)-\tilde a\right)u_1(t,x) \D x + \int_{\Omega \setminus \Omega_a}\left(a(x)-\tilde a\right)u_1(t,x) \D x,
\end{eqnarray*}
where $\Omega_a$ is defined in the expression \eqref{omega_a}. Consequently, using boundedness of $\rho_1$, boundedness of $a(x)$ as well as Lemma \ref{lem:decay2},  we obtain that
\begin{eqnarray*}
\int_{\Omega}\left(a(x)-\tilde a\right)u_1(t,x) \D x \:\: { \stackrel{t\rightarrow\infty}{\longrightarrow}}\:\: 0,
\end{eqnarray*}
and hence we conclude that system \eqref{sys:mass2} fulfills the assumptions of Lemma \ref{lem:dynamics}. Consequently, we obtain that the total mass of a solution of system \eqref{sys:reg} converges to a globally stable 
equilibrium, which is equal to the equilibrium of  the ordinary differential equations model  \eqref{sys:2eq} corresponding to the maximum value of the self-renewal parameter $\bar a$. 
} 
Thus, we have proven the assertion of Theorem \ref{MassConv}.
\end{proof}
\begin{figure}[h]
\begin{center}
\includegraphics[scale=0.6]{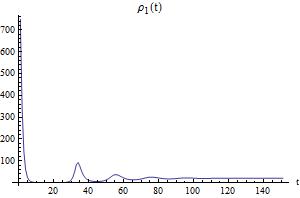}
\includegraphics[scale=0.6]{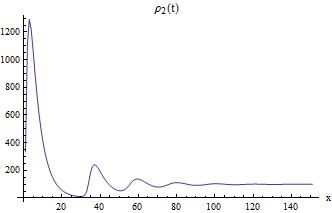}
\caption{Using the trapezoid rule to approximate the integral of $\rho_1(t),\rho_2(t)$, we observe numerically the convergence of the total mass to a constant value. The parameter set is the same as in Figure \ref{Fig1}.}
\end{center}
\end{figure}\label{Fig.3}

\section{Proof of the convergence result}\label{convergence}

{  Finally, we obtain the main assertion.
\begin{proof}[\bf of Theorem \ref{thm:concentration}]
{\\}
  Lemma \ref{lem:decay2} implies that the solutions of system \eqref{sys:reg} decay exponentially to zero in all points $x \notin \Omega_a$. We consider two cases (compare Assumptions \ref{Ass} (iv)):
 \begin{itemize}
 \item[(i)] $\Omega_a= \{ \bar x\} $:\\
Convergence to a stationary solution follows from the convergence of mass given by Theorem \ref{MassConv}.  { Hence, the solutions converge to measures concentrated at $\bar x$: 
$$u_i(t,\cdot)\mathcal L^1  \: \stackrel{t\rightarrow\infty}{\longrightarrow} c_i \delta_{\bar x}, \quad \text{for} \quad i=1,2,$$
where $\mathcal L^1$ denotes a one dimensional Lebesgue measure and $u_i(t,\cdot)\mathcal L^1$ is the measure which Radon-Nikodym derivative with respect to $\mathcal L^1$ is equal to $u$, $\delta_{\bar x}$ is a Dirac measure localised at $\bar x$ and  $c_i$, $i=1,2$, are the stationary masses, i.e. $c_1=\bar \rho_1=\frac{d}{p}\frac{2\bar a-1}{K}$ and $c_2=\bar \rho_2=\frac{2\bar a-1}{K}$.}\\

The convergence result can be understood in a suitable metric on the space of positive Radon measures.
We apply here the flat metric $\rho_F $, also known as the bounded Lipschitz distance \cite{Neunzert}. For completeness of presentation, the definition and basic properties of this metric are provided in Appendix.\\

{To estimate the distance between a solution $u_i(t,\cdot)$ and the stationary measure $c_i \delta_{\bar x}$, $i=1,2$,  we use the following inequality for the distance of two measures $\mu$ and $\nu$
 \begin{equation}\label{estimate_metric}
 \rho_{F}(\mu,\nu)\le \min\{\mu(\Omega),\nu(\Omega)\} W_1\left(\frac{\mu}{\mu(\Omega)},\frac{\nu}{\nu(\Omega)}\right) + |\mu(\Omega)-\nu(\Omega)|.
 \end{equation}
For the proof of this inequality we refer to \cite{CGU} and \newline \cite{JJAMC}. Here $W_1\left(\frac{\mu}{\mu(\Omega)},\frac{\nu}{\nu(\Omega)}\right)$ denotes the Wasserstein distance between two probabilistic measures; { see Appendix for the definition of the Wasserstein metric.}\\

{ We calculate, for $i=1,2$,
 \begin{equation}\label{estimate_metric_our}
 \rho_{F}\left(u_i(t,\cdot)\mathcal L^1,c_i \delta_{\bar x}\right)\le \min\{\rho_i,c_i)\} W_1\left( \frac{u_i(t,\cdot)\mathcal L^1}{\rho_i},\delta_{\bar x} \right) + |\rho_i-c_i |.
  \end{equation}

 The first term on the right hand-side of inequality \eqref{estimate_metric_our} can be estimated using the exponential estimates of Lemma \ref{asympt}. To show that it converges to zero we apply the Kantorovich-Rubinstein Theorem \cite{Villani1,Villani2} and use the equivalent definition of the Wasserstein metric given as the cost of optimal transport with the cost function $|x - y|$, i.e.
\begin{equation}\label{T_defi}
  W_1\left( \frac{\mu}{\mu(\Omega)}, \frac{\nu}{\nu(\Omega)}\right) : = \inf_{\gamma \in {\mathcal P}(\Omega)\times {\mathcal P}(\Omega)}
\int_{\Omega \times \Omega} |x - y|\;   \gamma(\D x,\D y),
\end{equation}
where $\gamma\in \Gamma\left( \frac{\mu}{\mu(\Omega)}, \frac{\nu}{\nu(\Omega)}\right) $ is a joint distribution (probabilistic measure) with the marginal distributions $ \frac{\mu}{\mu(\Omega)}$ and $\ \frac{\nu}{\nu(\Omega)}$, and where
 \begin{eqnarray*}
\Gamma\left(\frac{\mu}{\mu(\Omega)},\frac{\nu}{\nu(\Omega)}\right) =&& \Big\{
\gamma \in {\mathcal P}(\Omega \times \Omega)\gamma(B \times \Omega) = \frac{\mu(B)}{\mu(\Omega)},  \\
&&\;
\gamma(\Omega \times B) = \frac{\nu(B)}{\nu(\Omega)}, \;
B \in {\mathcal B}(\Omega)
\Big\}.
\end{eqnarray*}
is the family of all joint distributions with marginal distributions $ \frac{\mu}{\mu(\Omega)}$ and $\ \frac{\nu}{\nu(\Omega)}$.\\

\noindent  We estimate the difference between a normalised solution $$\pi_i(t):=\frac{u_i(t,\cdot)}{\rho_i(t)}\mathcal L^1$$ and its  limit  $\delta_{\bar x}$, $i=1,2$.
Using a joint distribution $\gamma_i =  \delta_{\bar x} \otimes \pi_i$, $i=1,2$, we obtain
\begin{equation}\label{W1est}
  W_1(\pi_i(t), \delta_{\bar x}) \leq  
\int_{\Omega} |\bar x - y|\;   \pi_i(t) (\D y),
\end{equation}
To show that the right-hand side of inequality \eqref{W1est} converges to zero, we define a set $\Omega_{a-{\varepsilon}}=\left \{x: a(x)>\bar a - \varepsilon \right\}$.  For $\varepsilon$ small enough, there exists $\tilde \varepsilon>0$ such that the set $\Omega_{a- \varepsilon}$ is contained in a $\tilde{\varepsilon}-$neighbourhood of $\Omega_{a}$, i.e. $\Omega_{a- \varepsilon} \in [\bar x- \tilde{\varepsilon}, \bar x + \tilde{\varepsilon}]$. By Lemma \ref{asympt}, $\pi_i(t)\left(\Omega \setminus [\bar x- \tilde{\varepsilon}, \bar x + \tilde{\varepsilon}]\right) \rightarrow 0$ for $t\rightarrow \infty$.
Therefore, we obtain
\begin{eqnarray*}
 W_1(\pi_i(t), \delta_{\bar x})&\leq& \int_{\Omega\setminus [\bar x- \tilde{\varepsilon}, \bar x + \tilde{\varepsilon}]} |\bar x - y|\;   \pi_i(t) (\D y) + \int_{[\bar x- \tilde{\varepsilon}, \bar x + \tilde{\varepsilon}]} |\bar x - y|\;   \pi_i(t) (\D y) \\
&\leq& \sup \limits_{x\in \Omega}  |\bar x - x| \pi_i(t)\left(\Omega \setminus [\bar x- \tilde{\varepsilon}, \bar x + \tilde{\varepsilon}]\right) + \tilde{\varepsilon}\rightarrow \tilde{\varepsilon}, \quad \text{for} \quad t\rightarrow \infty. 
 \end{eqnarray*}
 Since the above convergence holds for any $\tilde{\varepsilon}>0$, we conclude that
 $$\lim \limits_{t \rightarrow \infty} W_1(\pi_i(t), \delta_{\bar x}) = 0.$$
 
 Convergence of the second term in formula \eqref{estimate_metric_our} is due to Theorem \ref{MassConv}.  Hence, we obtain that  $$\lim \limits_{t \rightarrow \infty} \rho_{F}\left(u_i(\cdot,t)\mathcal L^1,c_i \delta_{\bar x}\right)=0.$$
}}
\\

\item[(ii)]  $\mathcal L^1 (\Omega_a)>0$:

If $\Omega_a$ is a set with positive measure, no singularities emerge due to the uniform boundedness of 
the total mass. In this case,  the solution tends to zero outside $\Omega_a$ and to a positive $L^1$-function on $\Omega_a$. Following Corollary \ref{MassDistr}, we conclude that the exact shape of the limit solution depends on the initial distribution. 
{
\item[(iii)]  If $\bar a = \max\limits_{x\in\overline\Omega} a(x) \leq \frac12$, then the solutions converge exponentially to zero, what is a consequence of equations \eqref{sys:reg}.  We estimate
$$\frac{d}{dt}\rho_1(t)\leq \left (\frac{1}{1+K \rho_2(t)}-1\right )p\rho_1(t) \leq -C \rho_1(t), $$
where $C= - \left(\frac{1}{1+K \min_{t\in[0,\infty)}\rho_2(t)}-1\right)p >0$, due to Lemma \ref{lem:1}. Hence, using the Gronwall inequality, we obtain the exponential decay to zero. Finally, convergence $\rho_2(t) \rightarrow 0$ as $t\rightarrow \infty$ follows from the estimate
$$\frac{d}{dt}\rho_2(t)\leq 2p \rho_1(t) -d \rho_2(t).$$
Since the solutions $(u_1,u_2)$ are nonnegative,  they converge to zero in $L^1(\Omega)$. }
\end{itemize}
\end{proof}
}

{
{\noindent Finally, we analyse the case with $\Omega_a$ consisting of two points and prove the co-existence and the extinction result.
\begin{proof}[\bf of Theorem \ref{concentration2}]
{\\}
(i) We investigate dynamics of the mass of a solution of system \eqref{sys:reg} around the points of $\Omega_a$. Let us assume that there exists a diffeomorphism $\Phi \in C^1(U_1)$, where $U_1$ is an open neighbourhood of $\bar x_1$, such that $\Phi(\bar x_1)=\bar x_2$ and $a(x)=a(\Phi(x))$ for all $x\in U_1$.
Using the explicit form of the solution \eqref{uexplicit} and the property $\Phi(\bar x_1)=\bar x_2$,
we obtain
\begin{equation}\label{balls1}
 \int \limits_{U_1} u_1(t,x) \D x = \int \limits_{U_1} u_1(t,\Phi(x)) \frac{u^0_1(x)}{u^0_1(\Phi(x))} \D x,
 \end{equation}
Changing variables on the right hand-side of \eqref{balls1} leads to
\begin{equation}\label{balls2}
 \int \limits_{U_1} u_1(t,x) \D x = \int \limits_{\Phi(U_1)} u_1(t,y) \frac{u^0_1(\Phi^{-1}(y))}{u^0_1(y)} J\Phi^{-1}(y) \D y,
 \end{equation}
where $J\Phi$ is Jacobian of the diffeomorphism $\Phi$.

Since $\frac{u^0_1(\Phi^{-1}(y))}{u^0_1(y)} J\Phi^{-1}(y)$ does not depend on time and is continuous with respect to $y$ and since $u(t,x)$ converges pointwise to zero outside $\Omega_a=\{\bar x_1, \bar x_2\}$ (see Lemma \ref{lem:decay2}), we obtain
\begin{equation}\label{balls3}
\lim\limits_{t\rightarrow +\infty} \int \limits_{U_1} u_1(t,x) \D x = \frac{u^0_1(\bar x_1)}{u^0_1(\bar x_2)} J\Phi^{-1}(\bar x_2)  \lim\limits_{t\rightarrow +\infty} \int \limits_{\Phi(U_1)} u_1(t,y) \D y,
 \end{equation}
Hence, the solution converges to a measure $c_{1,1}\delta_{\bar x_1} + c_{1,2}\delta_{\bar x_2}$ with strictly positive $c_{1,1}$ and $c_{1,2}$ such that 
\begin{equation}\label{MassProportion}
\frac{c_{1,1}}{c_{1,2}}=\frac{u^0_1(\bar x_1)}{u^0_1(\bar x_2)} J\Phi^{-1}(\bar x_2).
\end{equation} 
Since the total mass of $u_1$ is equal to $c_{1,1}+c_{1,2}=\bar \rho_1$, where $\bar \rho_1$ is given in Corollary \ref{Masses},  the constants $c_{1,1}$ and $c_{1,2}$ are uniquely determined. Relationship \eqref{MassProportion} indicates that the mass distribution between the different concentration points depends on the shape of the function $a(x)$ and on the initial data. 
 
(ii) Now, we consider the case where the mapping $\Phi$ defined above is only a homeomorphism and  $J\Phi^{-1}$ is continuous but  $J\Phi^{-1}(\bar x_2)=0$. Hence, equation \eqref{balls3} yields that $\lim\limits_{t\rightarrow +\infty} \int \limits_{U_1} u_1(t,x) \D x =0$, which implies that   the solution converges to a mass $c_{1,2}\delta_{\bar x_2}$  with $c_{1,2}=\bar \rho_1$.
\end{proof}

\begin{remark}
Continuity of $\frac{u^0_1(\Phi^{-1}(y))}{u^0_1(y)} J\Phi^{-1}(y)$ requires continuity of the initial data and strict positivity of $u^0_1$ on $\Omega_a$, which is reflected in the stronger assumptions of the theorem compared to Assumptions \ref{Ass}.
 \end{remark}
}

{

\section{Extension to initial data in the space of Radon measures}\label{Measures}

\noindent The phenomenon of mass concentration provides a motivation to consider the model in the space of positive Radon measures, as defined by the following equations
\begin{eqnarray}\label{sys:measures}
\frac{d}{d t} \mu_1(t)(B) &=& \int\limits_{B}\left (\frac{2a(x)}{1+K\rho_2(t)}-1\right)p  \mu_1(t)(\D x),\nonumber \\
\frac{d}{d t} \mu_2(t)(B) &=& \int\limits_{B}2\left(1-\frac{a(x)}{1+K\rho_2(t)}\right)p \mu_1(t) (\D x) - d\int\limits_{B} \mu_2(t)(\D x),
\end{eqnarray}
with
\begin{equation}\label{sys:rho}
\rho_i(t) = \int\limits_{\Omega}  \mu_i(t)(\D x), \quad i=1,2,
\end{equation}
 with the initial data
\begin{eqnarray}\label{sys:ini}
\mu_1(0) &=& \mu_1^0, \nonumber\\
\mu_2(0) &=& \mu_2^0,
\end{eqnarray}
where  $\mu_i^0$ are nonnegative Radon measures for $i=1,2$.  $x\in \Omega\subset \R^n$, for some $n\geq 1$, denotes the state of a cell and, for every Borel subset $B \subset \Omega$, 
$\mu_i(t)(B)=\int_{B} d\mu_i(t)$, $i=1,2$, are measures of cells in any of the states $x \in  B$ at
time $t$.  Variable $\rho_i$ denotes the mass of all cells from the $i-th$ compartment. 
Measures $\mu(t)$ are $C^1$ functions of time with values in the space of positive Radon measures with the total variation norm. Therefore, the time derivatives in equations \eqref{sys:measures} are understood as derivatives of the functions with values in a Banach space. \\

\noindent {Selection-mutation models in the spaces of positive Radon measures have been studied by many authors \cite{AMHF,AFT,BB,B,CCC,CA,DJMR}. In this context, convergence of the solutions with respect to the Prokhorov metric has been considered in Ref. \cite{AMHF}. For the relation between the Prokhorov metric and the Wasserstein distance used in our paper we refer to Ref.  \cite{GibbsSu}.}\\

\noindent Steps of the proof of Theorem \ref{thm:concentration} can be repeated for the measure-valued solutions with some modifications of the lemmas which rely on point-wise estimates of the quotients of solutions. Assuming that the initial data are measures such that $\mu_1^0$ is absolutely continuous with respect to $\mu_2^0$, Lemma \ref{Bound} can be reformulated for the model \eqref{sys:measures}-\eqref{sys:ini} by considering a Radon-Nikodym derivative
\begin{equation}\label{RadonNikodym}
\left(D_{\mu_2(t)}\mu_1(t)\right)(x) = \lim_{r \rightarrow 0^+}\frac{\mu_1(t)(B_{x,r})}{\mu_2(t)(B_{x,r})}
\end{equation}
instead of the point-wise quotients.\\

\noindent Next technical difficulty appears in Lemma \ref{asympt}. To show the asymptotic behaviour of the measure-valued solutions, we can apply the framework developed in Ref. \cite{BB}.  In the remainder of this section, we briefly discuss this extension.\\

\noindent The first equation of the model  \eqref{sys:measures}-\eqref{sys:ini} can be re-defined in the terms of a probabilistic measure modelling the frequency of a certain phenotype $x\in B$ in the population of mitotic cells $\mu_1$. It is given by the quotient 
$$\pi(t)(B)={\frac{\mu_1(t)(B)}{\mu_1(t)(\Omega)}},$$ 
where $B \subset \Omega$ is a Borel set, as defined before.\\

\noindent Using the equation for $\mu_1$, we obtain
\begin{equation}\label{eq_prob}
\frac{d}{d t}\pi(t)(B)=\frac{2p}{1+\rho_2(t)}\int\limits_{B}\left(a(x)-\int_{\Omega} a(\xi)\; \pi(t)(\D \xi) \right) \; \pi(t)(\D x).
\end{equation}
The model can be then formulated in the framework presented in the book by B\"urger \cite{B}.  Denoting the mean fitness by
\begin{equation}\label{meanfitness}
  \overline{\mathcal A}(t) = \frac{2p}{1+\rho_2(t)} \int_{\Omega} a(\xi)\; \pi(t)(\D \xi)
 \end{equation}
and the multiplication operator $\mathcal A(t)$ by
\begin{equation}\label{OperatorA}
\left(\mathcal A(t)\pi(t)\right)(B) =\frac{2p}{1+\rho_2(t)}\int\limits_{B}a(x) \pi(t)(\D x),
\end{equation}
we rewrite equation \eqref{eq_prob} as an ordinary differential equation in the space of Radon measures
\begin{equation}\label{eq_prob_abstract}
\frac{d}{dt}\pi(t) = \mathcal A(t) \pi(t) -  \overline{\mathcal A}(t) \pi(t).
\end{equation}
However, the obtained equation is more general than  the abstract equation in \cite{B}, due to the dependence of $\mathcal A$ on time.
Nevertheless, it holds
$$ \overline{\mathcal A}(t)=\left(\mathcal A(t) \pi(t)\right)(\Omega).$$
Using the form of the operator \eqref{OperatorA}, we rewrite it as a function of time $\alpha(t)=\frac{2p}{1+\rho_2(t)}$ multiplied by a time independent operator $\left(A \pi(t)\right)(B)=\int\limits_{B}a(x)\pi(t)(\D x)$,

\begin{equation}\label{alpha}
\mathcal A(t) =\alpha(t) A.
\end{equation}
This structure allows to follow the lines of \cite{BB} and focus on a differential equation given by
\begin{equation}\label{q-eq-t}
\frac{d}{dt}Q(t)=\mathcal A(t)Q(t).
\end{equation}
The structure assures that the family of operators ${\mathcal A}$ commutes. The operator ${\mathcal A}$ is bounded and it  generates a positive semigroup on the space of positive Radon measures ${\mathcal M^+}({\Omega})$.

Since $\alpha$ is a strictly positive and bounded function, due to the properties of $\rho_2$ shown in Lemma \ref{lem:1}, we can rescale time,  $s=\int_0^t \alpha(\xi)\D \xi$, and obtain a linear autonomous differential equation
\begin{equation}\label{q-eq-s}
\frac{d}{ds}Q(s)= AQ(s).
\end{equation}
Equivalence to a linear differential equation yields convergence of solutions to a solution $\pi(t)$ with the support concentrated on the set of maximal value of $a(x)$, $\bar a= \sup \limits_{x\in \Omega \cap {\rm supp}(\mu_1^0)}a(x)$. The latter result is the extension of our Lemma \ref{asympt} to the measure-valued solutions. \\
}}}

\noindent { In summary, by adapting the framework  developed in Ref. \cite{BB}, our results can be extended to the measure-valued solutions in the case of the model of the clonal evolution without mutations.   Asymptotic analysis carried out in \cite{BB} is based on the application of the infinite-dimensional version of the Perron-Frobenius Theorem, which is possible in the models with dynamics governed by an irreducible operator. The latter is the case in the models involving mutations described by an integral operator satisfying irreducibility conditions. That approach cannot be, however, directly applied to the extension of our model to the case with mutations. The difficulty is related to the estimates for the time dependent operator $\mathcal A$ defined in expression \eqref{OperatorA}, which rely on the equations for the ratios of  solutions in Lemma \ref{Bound}, or Radon-Nikodym derivatives  \eqref{RadonNikodym}, which cannot be established in the model with an additional nonlocal mutation operator. Therefore, including mutations in our model requires a different proof of the uniform boundedness and strict positivity of $\rho_2$ and extension of the analysis to the model with mutations remains an open question.
 }

{{{{\section{Discussion.}\label{discussion}

{ In this paper, a discrete multi-compartmental model of multiple cell lineages has been extended to a model coupling a two-stage differentiation structure with a continuous structure of phenotypes. The latter allows to investigate the role of the intra-cancer heterogeneity, including competition between healthy and cancer cells and dynamics of the multi-clonal structure of the system.\\

\noindent Based on recent analyses of the clones consisting of mutational variants in cancer \cite{Miller}, it follows that the dynamics of clone distributions may in many cases consist solely of change in relative frequencies of different clones. More specifically, the clones that have been dominant in the primary tumour, are out-competed by other clones in the relapsing or metastatic tumours, which  had low frequencies in the primary. The model in this paper provides a "mechanistic" explanation for these observations, which is also mathematically rigorous.\\

\noindent Asymptotic analysis of the proposed system of integro-differential equations suggests that the selection process may be governed by the cell's property of self-renewal that determines the fitness of each clone and ultimately leads to survival or extinction.}\\

\noindent Theorem \ref{thm:concentration} shows that, in a well-mixed cell production system, a negative nonlinear feedback such as that the one proposed in Ref. \cite{LanderB,Lander,MCSHJW}, leads to the selection of the subpopulation with the superior self-renewal potential. The assumption that the cell population is well-mixed leads to the nonlocal effect and is modelled using the integral term. This assumption reflects well the structure of the hematopoietic system. Consequently, our results suggest that the greater clonal heterogeneity observed in solid cancers than in blood cancers may be due to spatial effects of the cell-to-cell interactions. Additionally, Theorem \ref{concentration2} suggests some explanation of the co-existence of different clones having the same fitness. \\

\noindent The results stress the importance of self-renewal in cancer dynamics and allow concluding that slowly proliferating cancer cells with a high self-renewal potential are able to outcompete the cells that divide faster. It suggests an explanation of the clinical dynamics such as resistance to treatment. Importance of this observation in the context of the leukemia evolution, the response to chemotherapy and the dynamics of the disease relapses has been discussed  in Ref. \cite{StiehlBaranHoMarciniak}. The results obtained provide an explanation of the observed clonal selection in the acute myeloid leukemia in the course of the disease development and the relapse after chemotherapy reported by Ref. \cite{Ding}.
{ Recently,  fitting the AML model to patients' data has suggested that an increased self-renewal is correlated with a poor patient prognosis  \cite{StiehlBaranHoMarciniak2}.}
}}}

\section{Appendix}

\subsection{Flat metric}
We present here basic results concerning the space of positive Radon measures equipped with the flat metric $\rho_F $,  known also as the bounded Lipschitz distance \cite{Neunzert}.

\begin{definition}\label{defflatmetric} 
Let $\mu, \nu \in{\mathcal M^+}({\Omega})$.  The distance
function $\rho_F : {\mathcal M^+}({\Omega}) \times  {\mathcal M^+}({\Omega}) \rightarrow
[0, \infty)$ is defined by
\begin{equation}\label{Flatmetric}
  \rho_F(\mu, \nu) :=
\dst \sup \Big\{\int_{\Omega} \psi  d (\mu - \nu)
  \big| \: \psi \in C^1({\Omega}), \|\psi\|_{W^{1,\infty}} \leq 1 \Big\},
  \end{equation}
where $$\|\psi\|_{W^{1,\infty}} := \max\{\|\psi\|_{\infty},\|\partial_x \psi\|_{\infty}\}.$$
\end{definition}

\noindent The $\rho_F$ distance metrizes both weak* and narrow topologies on each tight subset of Radon measures with uniformly bounded total variation \cite{Schwartz,AFP}.

\begin{remark}
Every bounded Radon measure on a bounded set  $\Omega$ has an integrable first moment and hence the distance  $\rho_F$ is finite.
\end{remark}
\begin{proposition}\label{corollarflat}
 Flat metric satisfies the following properties:\\
\begin{itemize}
\item scale-invariance
\[
 { \rho_F(\theta\cdot\mu,\theta \cdot\nu)= \theta \rho_F(\mu,\nu).}\\
\]
\item  translation-invariance
\[
  \rho_F(T_{x}\mu,T_{x}\nu)=   \rho_F(\mu,\nu).
\]
\end{itemize}
\end{proposition}

\noindent Completeness of the space $\left(\mathcal{M}^+(\Omega), \rho_F\right)$ is the result of $\left(\mathcal{M}^+(\Omega),\rho_F\right)$ being a subspace of  $\left(W^{1,\infty}(\Omega)\right)^*$ 
and the equivalence of the flat metric convergence and weak* convergence in $\mathcal{M}^+(\Omega)$, which is complete with respect to weak* convergence. Inclusion $\left(\mathcal{M}^+(\Omega),\rho_F\right) \subset \left(W^{1,\infty}(\Omega)\right)^*$ is proven using a standard approximation argument for the test functions and Proposition \ref{corollarflat}.
\begin{eqnarray*}
		\rho_F(\mu,\nu) &=& \sup\left\{\int\limits_{\Omega} \psi\,\D 				(\mu-\nu)\left|\psi\in C^1(\Omega), \left\|\psi\right\|_{W^{1,			\infty}(\Omega)} \leq 1 \right. \right\} \\
		&=&{ \sup\left\{\frac{1}{\theta}\int\limits_{\Omega} \varphi\,			\D(\mu-\nu)\left|\varphi\in W^{1,\infty}(\Omega), \left\|					\varphi\right\|_{W^{1,\infty}(\Omega)} \leq \theta 						\right. \right\}} \\
		&=& \left\|\mu-\nu\right\|_{\left(W^{1,\infty}(\Omega)\right)^*}
\end{eqnarray*}\\
Thus the flat metric is the metric induced by the dual norm of $W^{1,\infty}(\Omega)$; see e.g. \cite{GLMC,GMC, MullerOrtiz, Zhidkov}.

\subsection{Wasserstein metric}

The Wasserstein metric  $W_1: {\mathcal P}(\Omega)\times {\mathcal P}(\Omega) \longrightarrow [0,\infty)$ in its dual representation is defined by
$$ \begin{array}{@{}rcl@{\hspace*{2mm}}l@{}l@{}l@{}l@{}}
  W_1\left( \frac{\mu}{\mu(\Omega)},\frac{\nu}{\nu(\Omega)} \right)& :=
& \dst \sup & \Big\{ & \dst \int_{\Omega} \; \psi  \:\: d \left( \frac{\mu}{\mu(\Omega)} - \frac{\nu}{\nu(\Omega)}\right)
\; & \Big|\; \psi \in C^1(\Omega), \;\: {\rm Lip} \: \psi \leq 1
&\Big\}.
\end{array}$$
For more information on the Wasserstein metric we refer to \cite{Villani1,Villani2}.

\begin{acknowledgements}
This paper resulted from  the Collaborative Research Center, SFB 873 `Maintenance and Differentiation of Stem Cells in Development and Disease''. Collaboration of  AM-C and PG  was supported by the grant of National Science Centre No. 6085/B/H03/2011/40. The authors thank Frederik Ziebell for help with numerical simulations illustrating the results of this work. The authors are greatly indebted to the Associate Editor and the Referees for the helpful comments. 
\end{acknowledgements}

\end{document}